\documentclass[a4paper,11pt,reqno]{amsart}

\usepackage{amssymb}
\usepackage{amsmath}
\newtheorem{definition}{Definition}
\newtheorem{thm}{Theorem}
\newtheorem{lem}[thm]{Lemma}
\newtheorem{prop}[thm]{Proposition}

\newtheorem{rmk}{Remark}

\begin{document}






\title{Steady 3D viscous compressible  flows
with adiabatic exponent $\gamma\in (1,\infty)$}
\author{P.I. Plotnikov}
\address{Lavrentyev Institute  of Hydrodynamics,
Lavrentyev pr. 15, 630090 Novosibirsk, Russia}
\email{plotnikov@hydro.nsc.ru}

\author{W. Weigant}
\address{Universitat Bonn, Endenicher Allee 60, D-53115 Bonn, Germany}
\email{w-weigant@t-online.de}


\begin{abstract}
The Navier-Stokes equations for compressible barotropic  flow in the
stationary three dimensional case are considered. It is assumed that
a fluid occupies a bounded domain and satisfies the no-slip boundary
condition. The existence of a weak solution  under the assumption
that the adiabatic exponent satisfies $\gamma>1$ is proved. These
results cover the cases of monoatomic, diatomic, and polyatomic
gases.
\end{abstract}


\maketitle 

\section{Introduction}
\label{inna}
\renewcommand{\theequation}{1.\arabic{equation}}
\setcounter{equation}{0}
 We deal with  a stationary  boundary value
problem for the compressible Navier-Stokes equations. It is assumed
that a compressible fluid occupies a bounded domain $\Omega\subset
\mathbb R^3$ with $C^2$ boundary. The state of the fluid is
characterized completely by the density $\varrho(x)\geq 0$ and the
velocity $\mathbf u(x)$. The governing equations represent two basic
principles of fluid mechanics:  the mass balance
\begin{subequations}\label{inna1}
\begin{equation}
\label{inna1b} \text{div~}(\varrho \mathbf
u)=0\quad\text{~in~}\Omega,
\end{equation}
and the balance of momentum
\begin{equation}\label{inna1a}
 \text{div~}(\varrho\, \mathbf u\otimes\mathbf u)
+\nabla p(\rho)= \text{div~} \mathbb S(\mathbf u)+\varrho\, \mathbf
f\quad\text{in~}\Omega.\end{equation} Here  $\mathbf f\in
L^\infty(\Omega)$ is a given vector field, the viscous stress tensor
$\mathbb S$ and the pressure $p$ are given by
\begin{equation}\label{inna2}
    \mathbb S(\mathbf u)=\mu\big(\nabla\mathbf{u}+\nabla\mathbf{u}^\top-
    \frac{2}{3}\text{div~}\mathbf u\mathbb I)
    +\nu
    \text{div~}\mathbf{u}\mathbb I, \quad p=\varrho^\gamma,
\end{equation}
the viscosity coefficients $\mu, \nu$ satisfy the inequalities
$\mu+\frac{1}{3}\nu>0$, $\mu>0$. Throughout  the paper, we assume
that $\gamma>1$ is an arbitrary constant. Notice that the standard
values of the adiabatic exponent $\gamma$ are 5/3 for the monoatomic
gas, between 9/11 and 7/5 for the diatomic gas, and between 1 and
4/3 for the polyatomic gas, see \cite{JN}.
 Equations \eqref{inna1b}--\eqref{inna1a}
should be supplemented with boundary conditions. We assume that the
velocity field satisfies the no-slip boundary condition on
$\partial\Omega$, and that the total mass of the fluid is
prescribed:
\begin{gather}\label{inna1ca}
    \mathbf u=0\quad\text{on~}\partial \Omega, \\\label{inna1cb}
     \int_\Omega \varrho\, dx=M.
\end{gather}
\end{subequations}
 Our goal is to prove that problem \eqref{inna1} has at least one
 weak solution.  By a
weak solution we mean a couple $(\mathbf u, \varrho)\in
W^{1,2}_0(\Omega)\times L^\gamma(\Omega)$ such that the integral
identities
\begin{equation}\label{inna6}
    \int_\Omega\big(\varrho\mathbf u\otimes\mathbf u:\nabla\boldsymbol
    \xi+p(\varrho) \text{div~}\boldsymbol\xi-\mathbb
    S(\mathbf u):\nabla\boldsymbol\xi+\varrho\mathbf f\cdot \boldsymbol\xi\big)\,
    dx=0,
\end{equation}
\begin{equation}\label{inna7}
    \int_\Omega \varrho\mathbf u\cdot\nabla\zeta\, dx=0
\end{equation}
hold for all vector fields $\boldsymbol\xi\in C^\infty_0(\Omega)$
and all $\zeta\in C^\infty(\Omega)$.

The first nonlocal  results concerning the mathematical theory of
compressible Navier-Stokes equations are due to P.-L. Lions. In
monograph \cite{lions1} he proved the weak continuity of the viscous
flux, which is the most important result in the mathematical theory
of viscous compressible flows, and established the existence of a
weak renormalized solution to problem \eqref{inna1} for all
$\gamma>5/3$. More recently, Novotn\'{y} \&
 Stra\v{s}craba, \cite{novotnystrascr}, employed the
 concept of oscillation defect measure developed in
  \cite{FEIRBOOK} to prove the existence result for all
 $\gamma>3/2$.
Plotnikov \& Sokolowski, \cite{PS2}, proved the existence of
renormalized solutions to the Dirichlet boundary value problem for
the compressible Navier-Stokes equations  for all $\gamma>4/3$ .
However, they replaced mass  conservation condition \eqref{inna1cb}
by a more restrictive integral condition. Finally, Frehse,
Steinhauer, and Weigant, \cite{FSW1}, proved the existence of weak
renormalized solutions to problem \eqref{inna1} for all
$\gamma>4/3$.

The better results were obtained for periodic structures and the
Neumann boundary value problem. Jiang~ \& Zhou, \cite{JZ1},
\cite{JZ2}, proved the existence of renormalized periodic solutions
to the compressible Navier-Stokes equations assuming $\gamma>1$. In
a  paper \cite{JN}, Jessl\'{e}~\& Novotn\'{y} proved that for every
$\gamma>1$, system \eqref{inna1a}-\eqref{inna1b} has a renormalized
solutions satisfying  the slip boundary conditions
$$
\mathbf u\cdot\mathbf n=0,\quad (\mathbb S(\mathbf u)\mathbf
n)\times\mathbf n=0\text{~~on~~}\partial\Omega.
$$
In the recent paper \cite{JNP} by Jessl\'{e},  Novotn\'{y}, and
Pokorn\'{y} these results were extended to the case of compressible
heat conducting fluid.

At the present time,   the existence of solutions to the classical
no-slip problem \eqref{inna1} was proved under  assumption
$\gamma>4/3$. Our goal is to relax this restriction  and to extend
the existence theory to the range  $\gamma>1$. We aim to prove the
following

\begin{thm}\label{inna9}
Let $\Omega$ be a bounded domain with $C^2$ boundary and let
$\gamma>1$. Then for every $\mathbf f\in L^\infty(\Omega)$ and $M>0$
problem \eqref{inna1} has a weak solution such that
\begin{equation}\label{inna8}
    \|\mathbf
    u\|_{W^{1,2}_0(\Omega)}+\|\varrho\|_{L^{q\gamma}(\Omega)}
    +\|\varrho|\mathbf u|^2\|_{L^s(\Omega)}\leq c,
\end{equation}
where the exponents  $q>1$ and $s>1$ depend only on $\gamma$, the
constant $c$ depends only on $\Omega$, $\mathbf f$, $\gamma$, $M$,
$\mu$, and $\nu$.

\end{thm}

 Notice that the question of existence of solutions to
problem \eqref{inna1} is closely related to the questions of
boundedness and compactness of the set of solutions to compressible
Navier-Stokes equations corresponding to various pressure functions.
Indeed, we can approximate the pressure function $p(\varrho)$
 with a fast growing artificial pressure $p_\epsilon$.
 Then we  obtain a
sequence of solutions
 $(\mathbf u_\epsilon, \varrho_\epsilon)$ to the regularized problem
  with $p$ replaced by $p_\epsilon$. After this we have to prove that
  the solutions $(\mathbf u_\epsilon, \varrho_\epsilon)$
 have uniformly bounded energies.
 Hence, passing to a subsequence, if necessary, we can assume that
 $(\mathbf u_\epsilon, \varrho_\epsilon)$ converge weakly in
 the energy space to some limit functions $(\mathbf u, \varrho)$.
Finally, we have to prove that the limit  satisfies equations
\eqref{inna1}.
 To give a rigorous
  meaning to the above discussion, we take the approximation of the pressure function in the form
\begin{equation}\label{inna4}
    p_\epsilon(\varrho)=\varrho^\gamma+\epsilon \varrho^4,
\end{equation}
and consider the regularized problem
\begin{subequations}\label{inna5}
\begin{equation}\label{inna5a}
 \text{div~}(\varrho\, \mathbf u\otimes\mathbf u)
+\nabla p_\epsilon(\rho)= \text{div~} \mathbb S(\mathbf u)+\varrho\,
\mathbf f\quad\text{in~}\Omega,\end{equation}
\begin{equation}
\label{inna5b} \text{div~}(\varrho \mathbf
u)=0\quad\text{~in~}\Omega,
\end{equation}
\begin{equation}\label{inna5c}
    \mathbf u=0\quad\text{on~}\partial \Omega, \quad
     \int_\Omega \varrho\, dx=M.
\end{equation}
\end{subequations}
Without loss of generality we can assume that
\begin{equation}\label{inna3}
    1<\gamma\leq 2.
\end{equation}
It is known, see \cite{lions1}, that for every $\epsilon\in (0,1]$,
problem \eqref{inna5} has at least one weak solution $(\mathbf
u_\epsilon, \varrho_\epsilon)\in W^{1,2}_0(\Omega)\times
L^8(\Omega)$ satisfying integral identities
\eqref{inna6}-\eqref{inna7} with $p$ replaced by $p_\epsilon$.

The proof of Theorem \ref{inna9} contains two contributions. The
first states a priori estimates for solutions of the regularized
equations. The second establishes compactness properties for these
solutions. The question of compactness of solutions to regularized
equations was investigated thoroughly in monographs \cite{lions1,
novotnystrascr} and papers \cite{JN,PS2} . It is known that if weak
solutions $(\mathbf u_\epsilon, \varrho_\epsilon)$ to problem
\eqref{inna5} satisfy the inequality
\begin{equation}\label{inna10}
    \|\mathbf
    u_\epsilon\|_{W^{1,2}_0(\Omega)}+\|p_\epsilon(\varrho_\epsilon)\|_{
    L^{q}(\Omega)}+\|\varrho_\epsilon|\mathbf u_\epsilon|^2\|_{L^s(\Omega)}\leq c,
\end{equation}
 then  after passing to a subsequence we can assume that these
solutions converge weakly in $W^{1,2}_0(\Omega)\times L^{\gamma
q}(\Omega)$ to a weak solution to problem \eqref{inna1}. Hence
Theorem \ref{inna9} will be proved if we prove  the following

\begin{thm}\label{inna11}
Let $\Omega$ be a bounded domain with $C^2$ boundary, $\mathbf f\in
L^\infty(\Omega)$, $\gamma>1$, and   $M>0$. Furthermore assume that
$(\mathbf u_\epsilon, \varrho_\epsilon)\in W^{1,2}_0\times
L^8(\Omega)$, $\epsilon\in (0,1]$, are weak solutions to problem
\eqref{inna5}. Then there are exponents $q>1$, $s>1$, depending only
on $\gamma$, and  a constant $c$, depending only on $\Omega$,
$\mathbf f$,$M$, $\gamma$,  and $\mu, \nu$, such that these
solutions satisfy  inequality \eqref{inna10}.
\end{thm}
The rest of the paper is devoted to the proof of  Theorem
\ref{inna11}.

\section{Notation and definitions. Auxiliary Propositions}\label{anna}
\renewcommand{\theequation}{2.\arabic{equation}}
\setcounter{equation}{0}

{\bf Notation.} For every $\gamma>1$ we denote by
 $\theta$, $\beta$, $s$,  and $q$  the quantities
\begin{equation}\label{anna100}\begin{split}
\theta=\frac{1}{8}(1-\gamma^{-1}), \quad
    \beta=\frac{3\big(1-8\theta^2\big)}
    {2\big(3-8\theta^2\big)},\\
s=1+2\theta^2, \quad q=1+\frac{\beta(s-1)}{\beta+(1-\beta)s}.
\end{split}
\end{equation}
It is easily seen that
\begin{equation}\label{anna101}
    0<\theta<1/8, \quad 0<\beta<1/2, \quad 1<s<33/32.
\end{equation}
Further the signed distance function $d(x)$  is given by
\begin{equation}\label{anna102}
d(x)=\text{~dist~}(x,\partial\Omega)\text{~~for~~} x\in
\bar{\Omega}, \quad
d(x)=-\text{~dist~}(x,\partial\Omega)\text{~~for~~} x\in\mathbb
R^3\setminus\Omega.
\end{equation}
For every $c>0$ denote by $\mathcal A_c$ and $\Omega_c$ the
annuluses
\begin{equation}\label{anna103}
A_c=\{x\in \mathbb R^3:\,\, \text{dist~}(x,\partial\Omega)<c\},
\quad \Omega_c=A_c\cap \Omega.
\end{equation}
 Since $\partial \Omega \in
C^2(\Omega)$, there is $t>0$, depending only on $\Omega$, such that
\begin{equation}\label{anna104}
d\in C^2(\bar {\mathcal A_{2t}}), \quad |\nabla d(x)|=1\text{~~in
~~}\bar A_{2t}.
\end{equation}
See  \cite[chap. 14.6]{GT} for the proof.
\begin{definition}\label{anna105} Further, the notation $\varphi$ stands for the function
 $\varphi: \overline{\mathcal
A_{2t}\cup\Omega}\to \mathbb R$ with the properties:
\begin{itemize}
\item[1.]$\varphi\in C^2(\overline{\mathcal
A_{2t}\cup\Omega})$, ~~ $\varphi(x)=d(x)$ in $\mathcal A_{2t}$,
\item[2.] there is $k>0$ such that $\varphi>k$ in
$\Omega\setminus\Omega_{2t}$.
\end{itemize}
Notice that $\varphi$ is  positive in $\Omega$. The existence of
such a function obviously follows from the Whitney extension
theorem.
\end{definition}
\begin{rmk}\label{anna200} Recall that our goal is to obtain the a
priori estimates of solutions to regularized equations
\eqref{inna5}. These estimates depend on the flow domain, the
constitutive law, and the parameters in equations. Further, we
denote by $c_e$ general constants depending only on $\Omega$, $\mu$,
$\nu$, $\gamma$, $M$,  and  $\|\mathbf f\|_{L^\infty(\Omega)}$.
\end{rmk}

{\bf Auxiliary Lemmas.} In this section we prove two technical
lemmas. The first constitutes the properties of solutions to the
regularized problem.
\begin{lem}\label{anna5}
Let $\epsilon\in (0,1]$. Then problem \eqref{inna5} has a weak
renormalized solution $(\mathbf u, \varrho)\in
W^{1,2}_0(\Omega)\times L^8(\Omega)$ with the following properties.
The integral identities
\begin{equation}\label{anna7}
     \int_\Omega\big(\varrho\mathbf u\otimes\mathbf u:\nabla\boldsymbol
    \xi+p_\epsilon(\varrho) \text{div~}\boldsymbol\xi-\mathbb
    S(\mathbf u):\nabla\boldsymbol\xi+\varrho\mathbf f\cdot \boldsymbol\xi\big)\,
    dx=0,
\end{equation}
\begin{equation}\label{anna9}
 \int_\Omega \big(\psi(\varrho)\mathbf u\cdot\nabla\zeta+
\zeta (\varrho\psi'(\varrho)-\psi(\varrho))\text{\rm div~}\mathbf
u\big)\, dx=0
\end{equation}
hold for all $\boldsymbol\xi\in W^{1,2}_0(\Omega)$, $\zeta\in
C^\infty(\Omega)$, and for all functions $\psi\in C^1[0, \infty)$
satisfying the condition
$$
|\psi(s)|+|s\psi'(s)|\leq c (1+|s|^4), \quad s\in [0,\infty).
$$
Moreover, we have
\begin{equation}\label{anna11}
     \int_\Omega\mathbb
    S(\mathbf u):\nabla\mathbf u\, dx=\int_\Omega\varrho\mathbf f\cdot \mathbf
    u\,
    dx.
\end{equation}
\end{lem}
\begin{proof}
The existence of solutions satisfying \eqref{anna7} and
\eqref{anna9} was proved in \cite{lions1}. It remains to prove
\eqref{anna11}. Set
$$
\psi(\varrho)=\frac{1}{\gamma-1}\varrho^\gamma+\frac{\epsilon}{3}\varrho^4.
$$
Obviously we have
$\varrho\psi'(\varrho)-\psi(\varrho)=p_\epsilon(\varrho) $.
Substituting $\psi(\varrho)$ and $\zeta=1$ into \eqref{anna9} we
arrive at the identity
$$
\int_\Omega p_\epsilon(\varrho)\text{div~}\mathbf u\, dx=0.
$$
It follows from this and identity \eqref{anna7} with
$\boldsymbol\xi=\mathbf u$ that
\begin{equation}\label{anna12}
     \int_\Omega\big(\varrho\mathbf u\otimes\mathbf u:\nabla\mathbf u-\mathbb
    S(\mathbf u):\nabla\mathbf u+\varrho\mathbf f\cdot \mathbf u\big)\,
    dx=0.
\end{equation}
Next we have
$$
\int_\Omega\varrho\mathbf u\otimes\mathbf u:\nabla\mathbf u\,
dx=\frac{1}{2}\int_\Omega \varrho\mathbf u\nabla(|\mathbf u|^2)\,
dx.
$$
Since the embedding $W^{1,2}_0(\Omega)\hookrightarrow L^6(\Omega)$
is bounded,  we have $|\mathbf u|^2\in W^{1, 3/2}(\Omega)$ and
$\varrho|\mathbf u|^2\in L^2(\Omega)$. Hence there is a sequence
$\zeta_n\in C^1(\Omega)$ such that
\begin{equation}\label{anna13}
    \nabla\zeta_n \to \nabla(|\mathbf u|^2)\text{~~in~~}
    L^{3/2}(\Omega).
\end{equation}
On the other hand, the Cauchy inequality  implies
$$
(\varrho|\mathbf u|)^3\leq  \varrho^6+|\mathbf u|^6\in L^1(\Omega).
$$
From this  we conclude that $\varrho\mathbf u\in L^3(\Omega)$.
Notice that $\varrho\mathbf u$ satisfies the integral identity
$$
\int_\Omega \varrho\mathbf u\cdot \nabla\zeta_n\, dx=0.
$$
Letting $n\to \infty$ and using relation \eqref{anna13} we arrive at
\begin{equation}\label{anna13a}
\int_\Omega\varrho\mathbf u\otimes\mathbf u:\nabla\mathbf u\,
dx=\frac{1}{2}\int_\Omega \varrho\mathbf u\nabla(|\mathbf u|^2)\,
dx=0.
\end{equation}
Inserting \eqref{anna13a} into \eqref{anna12} we obtain the desired
identity \eqref{anna11}.
\end{proof}
The second lemma is of general character. It is known, see
\cite{PS1, PS2}, that the boundedness of the Green potential of a
Borel measure $\sigma$ implies the continuity  of the embedding
 $W^{1,2}_0(\Omega)\hookrightarrow L^2(\Omega, d\sigma)$. This
 fact is a straightforward consequence of the
Mazja-Adams embedding theorem, see \cite{AH}. We give an elementary
proof of this result in a particular case.
\begin{lem}\label{sura7e}
Let $\Omega\in \mathbb R^3$ be a bounded domain with $C^2$ boundary.
Let $f\in L^2(\Omega)$ satisfy
$$
f\geq 0, \quad \int_\Omega f(x)|x-x_0|^{-1}\, dx \leq E \text{~~for
all~~}x_0\in \Omega.
$$
 Then there is $c>0$, depending only on $\Omega$, such that
\begin{equation}\label{sura7ee}
\int_\Omega | u|^2 f\, dx \leq c E\,\|
u\|_{W^{1,2}_0(\Omega)}^2\text{~~for all~~}u\in W^{1,2}_0(\Omega).
\end{equation}
\end{lem}
\begin{proof}
Let $h\in W^{2,2}(\Omega)$ be a solution to the boundary  value
problem
\begin{equation}\label{sura9}
    -\Delta h=f\text{~~in~~}\Omega, \quad h=0\text{~~on~~}\partial
    \Omega.
\end{equation}
This solution has the representation
$$
h(x_0)=\int_\Omega G(x,x_0)f(x)\,dx,
$$
where the Green function  $G(x,x_0)$   admits the estimate $
G(x,x_0)\leq c|x-x_0|^{-1}. $ We thus get
\begin{equation}\label{sura10}
    \|h\|_{L^\infty(\Omega)}\leq c E.
\end{equation}
Set $$ K= \int_\Omega | u|^2 f\, dx.$$ We have
$$
K=\int_\Omega | u|^2(-\Delta h)\,dx=2\int_\Omega( u\nabla
u)\cdot\nabla h\, dx.
$$
Applying the Cauchy inequality we obtain
\begin{equation}\label{sura9}
    K\leq \| u\|_{W^{1,2}_0(\Omega)}\Big(\int_\Omega|
    u|^2|\nabla h|^2\, dx\Big)^{1/2}.
\end{equation}
Let us estimate the integral in the right hand side. Integration by
parts gives
\begin{equation}\begin{split}\label{kuka1}
\int_\Omega|
    u|^2|\nabla h|^2\, dx=-2\int_\Omega h\big( u \nabla
    u)\cdot\nabla h\, dx+\int_\Omega h(-\Delta h)| u|^2\,
    dx\\\leq  \|h\|_{L^\infty(\Omega)}\Big(\int_\Omega | u| |\nabla
    u||\nabla h|\, dx+K\Big)\leq
   c E \Big(\int_\Omega | u| |\nabla
    u||\nabla h|\, dx+K\Big).
\end{split}\end{equation}
In turn,  the Cauchy inequality implies
$$
c E \int_\Omega | u| |\nabla
    u||\nabla h|\, dx\leq \frac{1}{2}\int_\Omega|
    u|^2|\nabla h|^2\, dx+c E^2\|
    u\|_{W^{1,2}_0(\Omega)}^2.
$$
Inserting this inequality into \eqref{kuka1} we arrive at
\begin{equation*}\begin{split}
\int_\Omega|
    u|^2|\nabla h|^2\, dx\leq
E^2\|
    u\|_{W^{1,2}_0(\Omega)}^2+ cEK.
\end{split}\end{equation*}
Inserting this inequality into \eqref{sura9} we obtain
\begin{equation*}
K\leq c\| u\|_{W^{1,2}_0(\Omega)}^2 E+ c \|
u\|_{W^{1,2}_0(\Omega)}(EK)^{1/2}.
\end{equation*}
Noting that
$$
 c
\| u\|_{W^{1,2}_0(\Omega)}(EK)^{1/2} \leq \frac{1}{2} K +c\|
u\|_{W^{1,2}_0(\Omega)}^2 E
$$
we finally obtain
$$
K\leq c\| u\|_{W^{1,2}_0(\Omega)}^2 E.
$$
\end{proof}
\section{Estimates near the boundary. Pressure estimates}
\renewcommand{\theequation}{3.\arabic{equation}}
\setcounter{equation}{0}The remarkable property of the compressible
Navier-Stokes equations is that the normal component of the energy
tensor $2^{-1}\varrho\mathbf u\otimes\mathbf u+p_\epsilon
(\varrho)\mathbf I$ is  small by comparison to its tangent component
in a neighborhood of $\partial\Omega$, see \cite{PS1,PS2}.  The
following  lemma is a refined version of  this result.
\begin{lem}\label{lena1} Let  a solution
  $(\mathbf u, \varrho)\in W^{1,2}_0(\Omega)\times
L^8(\Omega)$ to problem \eqref{inna5} be defined by Lemma
\ref{anna5}.
 Then
\begin{equation}\label{lena2}\begin{split}
    \|p_\epsilon(\varrho)\varphi^{-\beta}\|_{L^1(\Omega)}
    +\|\varrho (\mathbf
    u\cdot\nabla\varphi)^2\varphi^{-\beta}\|_{L^1(\Omega)} \leq\\
     c_e\big(1+\|\mathbf u\|_{W^{1,2}_0(\Omega)}+
     \|p_\epsilon(\varrho)\varphi^{1-\beta}\|_{L^1(\Omega)}+\|
     \varrho|\mathbf
     u|^2\varphi^{1-\beta}\|_{L^1(\Omega)}\big),
\end{split}\end{equation}
 where $\beta$ is given by \eqref{anna100} and $\varphi$ is given by Definition
 \ref{anna105}.\end{lem}
\begin{proof} Recall that $\beta\in (0,1/2)$. Introduce the vector field
\begin{equation}\label{lena3}
    \boldsymbol\xi(x)=\varphi^{1-\beta}(x)\, \nabla\varphi(x), \quad
    x\in \Omega.
\end{equation}
Obviously we have
\begin{equation*}
    \nabla\boldsymbol\xi=\varphi^{1-\beta}\nabla^2\varphi+(1-\beta)\varphi^{-\beta}
    \nabla\varphi\otimes \nabla\varphi.
\end{equation*}
It follows from \eqref{anna101} that
\begin{equation*}
    |\varphi^{1-\beta}\nabla^2\varphi|\leq c, \quad
    |\varphi^{-\beta}
    \nabla\varphi\otimes \nabla\varphi|\leq c\varphi^{-\beta},
\end{equation*}
and hence $\boldsymbol\xi \in W^{1,r}_0(\Omega)$ for all $r\in
[1,1/\beta)$. Substituting $\boldsymbol\xi$ in identity
\eqref{anna7} we obtain
\begin{equation}\begin{split}\label{lena4a}
     \int_\Omega\big(\varrho\mathbf u\otimes\mathbf u:\nabla\boldsymbol
    \xi+p_\epsilon(\varrho) \text{div~}\boldsymbol\xi\big)\, dx=\int_\Omega
    \big(\mathbb
    S(\mathbf u):\nabla\boldsymbol\xi\, dx-\varrho\mathbf f\cdot \boldsymbol\xi\big)\,
    dx\leq\\
    c\Big(\int_\Omega | S(\mathbf u)|^2\, dx\Big)^{1/2}+c\int_\Omega
    \varrho\, dx\leq c (1+\|\mathbf u\|_{W^{1,2}_0(\Omega)}).
\end{split}\end{equation}
On the other hand, we have
\begin{equation*}
p_\epsilon\text{div~}\boldsymbol\xi=(1-\beta)p_\epsilon\varphi^{-\beta}
|\nabla\varphi|^2+p_\epsilon \varphi^{1-\beta}\Delta\varphi\geq
2^{-1}\varphi^{-\beta}p_\epsilon-c \varphi^{1-\beta}p_\epsilon
\end{equation*}
and
\begin{equation*}\begin{split}
\varrho\mathbf u\otimes\mathbf u:\nabla\boldsymbol
    \xi=(1-\beta)\varphi^{-\beta}\varrho (\mathbf
    u\cdot\nabla\varphi)^2+\varphi^{1-\beta}\varrho
    u_iu_j\partial_i\partial_j\varphi\geq
    \\ (1-\beta)\varphi^{-\beta}\varrho (\mathbf
    u\cdot\nabla\varphi)^2-c\varphi^{1-\beta}\varrho|\mathbf u|^2.
\end{split}\end{equation*} Inserting  these inequalities into \eqref{lena4a}
we obtain \eqref{lena2}.
\end{proof}
 The proof that the pressure function is
integrable with some exponent greater than one is the essential part
of the mathematical analysis of compressible viscous flows. The
following lemma establishes this result for a weighted pressure
function.

\begin{lem}\label{lena4}
Let  a solution
  $(\mathbf u, \varrho)\in W^{1,2}_0(\Omega)\times
L^8(\Omega)$ to problem \eqref{inna5} be given by Lemma \ref{anna5},
and $(\beta,s)$ be given by \eqref{anna100}. Then
\begin{equation}\label{lena5}\begin{split}
   \|p_\epsilon(\varrho)\varphi^{1-\beta}\|_{L^{s}(\Omega)}
   \leq\\ c_e \big( 1
   +\|p_\epsilon(\varrho)\varphi^{-\beta}\|_{L^{1}(\Omega)}
+\|\varrho|\mathbf u|^2\varphi^{1-\beta}\|_{L^s(\Omega)}  +
\|\mathbf u\|_{W^{1,2}_0(\Omega)}\big).
\end{split}\end{equation}

\end{lem}
\begin{proof}
Choose an arbitrary function $g\in L^{s/(s-1)}(\Omega)$. It follows
from the Bogovskii lemma that the problem
\begin{equation}\label{lena6}\begin{split}
\text{div~}\boldsymbol\omega=g-\frac{1}{|\Omega|}\int_\Omega g\,
dx\text{~~in~~}\Omega, \quad
\boldsymbol\omega=0\text{~~on~~}\partial\Omega,
\end{split}\end{equation}
has a solution $\boldsymbol\omega\in W^{1,s/(s-1)}_0(\Omega)$
satisfying the inequality
\begin{equation}\label{lena7}
    \|\boldsymbol\omega\|_{W^{1,s/(s-1)}(\Omega)}\leq
    c\|g\|_{L^{s/(s-1)}(\Omega)}.
\end{equation}
Since $s/(s-1)>3$, the embedding $
W^{1,s/(s-1)}_0(\Omega)\hookrightarrow C(\bar\Omega)$ is bounded. It
follows from this and general properties of $
W^{1,s/(s-1)}_0(\Omega)$ that
\begin{equation}\label{lena8}
\|\boldsymbol\omega\|_{L^{\infty}(\Omega)}+
\|\varphi^{-1}\boldsymbol\omega\|_{L^{s/(s-1)}(\Omega)}\leq
    c\|g\|_{L^{s/(s-1)}(\Omega)}.
\end{equation}
Introduce the vector field
\begin{equation}\label{lena9}
    \boldsymbol\xi(x)=\varphi(x)\boldsymbol\omega(x), \quad
    x\in\Omega.
\end{equation}
It is easily seen that
\begin{equation*}
   \partial_j\xi_i=\varphi^{1-\beta}H_{ij},
\text{~~where~~}
    H_{ij}=\partial_j\omega_i+(1-\beta)\varphi^{-1}
   \partial_j\varphi\,\,\omega_i.
\end{equation*}
Inequalities  \eqref{lena7} and \eqref{lena8} imply
\begin{equation}\label{lena10}
 \|H_{ij}\|_{L^{s/(s-1)}(\Omega)}\leq
    c\|g\|_{L^{s/(s-1)}(\Omega)}.
\end{equation}
In particular, we have
\begin{equation}\label{lena11}
 \|\boldsymbol \xi\|_{W^{1,s/(s-1)}(\Omega)}+
 \|\boldsymbol\xi\|_{L^{\infty}(\Omega)}\leq
    c\|g\|_{L^{s/(s-1)}(\Omega)}.
\end{equation}
Substituting $\boldsymbol\xi$ into \eqref{anna7} we arrive at
\begin{equation}\label{lena12}\begin{split}
\int_\Omega p_\epsilon(\varrho)\text{~div~}\boldsymbol\xi\, dx=
\int_\Omega\big(\mathbb S(\mathbf u):\nabla\boldsymbol\xi
-\varrho\mathbf f\cdot \boldsymbol\xi\big)\,dx- \int_\Omega
\varphi^{1-\beta}\varrho H_{ij}u_iu_j\,dx\\
\leq c\|\mathbf
u\|_{W^{1,2}_0(\Omega)}\|\boldsymbol\xi\|_{W^{1,2}_0(\Omega)}
+c\|\boldsymbol\xi\|_{L^\infty(\Omega)}+c\|\varphi^{1-\beta}\varrho|\mathbf
u|^2\|_{L^s(\Omega)}\|\mathbf
H\|_{L^{s/(s-1)}(\Omega)}\\
\leq c\big(1+\|\mathbf
u\|_{W^{1,2}_0(\Omega)}+\|\varphi^{1-\beta}\varrho|\mathbf
u|^2\|_{L^s(\Omega)}\big)\|g\|_{L^{s/(s-1)}(\Omega)}.
\end{split}\end{equation}
On the other hand, we have
\begin{equation*}\begin{split}
\int_\Omega p_\epsilon(\varrho)\text{~div~}\boldsymbol\xi\, dx=
\int_\Omega \varphi^{1-\beta}p_\epsilon(\varrho)g\,
dx-\frac{1}{|\Omega|}\int_\Omega g\, dx \int_\Omega
\varphi^{1-\beta}p_\epsilon(\varrho)\, dx+\\
(1-\beta)\int_\Omega\varphi^{-\beta}p_\epsilon(\varrho)\boldsymbol\omega\cdot\varphi\,
dx.
\end{split}\end{equation*}
Next, inequality  \eqref{lena8} implies
\begin{equation*}
\Big|\int_\Omega\varphi^{-\beta}p_\epsilon(\varrho)\boldsymbol\omega\cdot\varphi\,
dx\Big|\leq
c\|g\|_{L^{s/(s-1)}(\Omega)}\int_\Omega\varphi^{-\beta}p_\epsilon(\varrho)\,
dx,
\end{equation*}
which yields
\begin{equation*}\begin{split}
\int_\Omega p_\epsilon(\varrho)\text{~div~}\boldsymbol\xi\, dx\geq
\int_\Omega \varphi^{1-\beta}p_\epsilon(\varrho)g\, dx-\\c
\Big(\int_\Omega \varphi^{1-\beta}p_\epsilon(\varrho)\, dx+
\int_\Omega\varphi^{-\beta}p_\epsilon(\varrho)\,
dx\Big)\|g\|_{L^{s/(s-1)}(\Omega)}\geq\\ \int_\Omega
\varphi^{1-\beta}p_\epsilon(\varrho)g\, dx-c \Big(
\int_\Omega\varphi^{-\beta}p_\epsilon(\varrho)\,
dx\Big)\|g\|_{L^{s/(s-1)}(\Omega)}. \end{split}\end{equation*}
Combining this result with \eqref{lena12} we finally arrive at the
inequality
\begin{equation*}\begin{split}
\int_\Omega \varphi^{1-\beta} p_\epsilon(\varrho)g\, dx \leq
c\big(1+\|\mathbf
u\|_{W^{1,2}_0(\Omega)}+\|\varphi^{1-\beta}\varrho|\mathbf
u|^2\|_{L^s(\Omega)}\big)\|g\|_{L^{s/(s-1)}(\Omega)}+\\
c\|p_\epsilon(\varrho)\varphi^{-\beta}\|_{L^{1}(\Omega)}
\,\|g\|_{L^{s/(s-1)}(\Omega)},
\end{split}\end{equation*}
which  yields \eqref{lena5}.
\end{proof}
\section{Quantities $A$ and $B$}
\renewcommand{\theequation}{4.\arabic{equation}}
 Introduce the quantities
\begin{equation}\label{anna3}
    A=\int_\Omega\varrho |\mathbf u|^{2(2-\theta)}\varphi^{2\beta}\,
    dx,\quad B=\int_\Omega \varrho^\gamma\varphi^{-\beta}\,
    dx.
\end{equation}

\begin{lem} Let $(\varrho, \mathbf u)\in W^{1,2}_0(\Omega)\times L^8(\Omega)$
be a weak  solution to problem \eqref{inna5}, and let $(\theta,
\beta, s)$ be given by relations \eqref{anna100}. Then
\begin{subequations}\label{anna15}
\begin{gather}\label{anna15a}
    \|\mathbf u\|_{W^{1,2}_0(\Omega)}\leq c_e
    A^{\frac{1}{4(2-\theta)}}\,
    B^{\frac{1}{2(2-\theta)(2\gamma-1)}},\\\label{anna15b}
    \|\varrho|\mathbf u|^2\varphi^{1-\beta}\|_{L^s(\Omega)}\leq c_e
    A^{\frac{1}{2-\theta}}\,
    B^{\frac{\theta}{(2-\theta)(2\gamma-1)}}.
\end{gather}
\end{subequations}
\end{lem}
\begin{proof}
Integral identity  \eqref{anna11}  implies
$$
\int_\Omega \big(\mu|\nabla\mathbf
u|^2+(\nu+\frac{\mu}{3})|\text{div~}|^2\big)\,
dx=\int_\Omega\varrho\mathbf u\cdot\mathbf f\, dx\leq
c_e\|\varrho\mathbf u\|_{L^1(\Omega)}.
$$
Since $\mu>0$ and  $\mu/3+\nu>0$, we have
\begin{equation}\label{anna16}
\|\mathbf u\|_{W^{1,2}_0(\Omega)}^2\leq c_e \|\varrho\mathbf
u\|_{L^1(\Omega)}.
\end{equation}
Introduce the quantities
\begin{equation}\label{anna17}\begin{split}
    \alpha_1=\frac{1}{2(2-\theta)},\quad
    \alpha_2=\frac{1}{(2-\theta)(2\gamma-1)}, \\
    \alpha_3=\frac{\gamma-1}{(2-\theta)(2\gamma-1)}, \quad
    \alpha_4=\frac{4\gamma-3-2\theta(2\gamma-1)}{2(2-\theta)(2\gamma-1)}.
\end{split}\end{equation} It follows from \eqref{anna100}
that $\alpha_i$ are positive and
$$
\alpha_1+\alpha_2+\alpha_3+\alpha_4=1, \quad
\alpha_1+\gamma\alpha_2+\alpha_4=1, \quad
2\beta\alpha_1-\beta\alpha_2- 2\beta\alpha_3=0.
$$
We thus get
$$
|\varrho\mathbf u|=(\varrho|\mathbf
u|^{2(2-\theta)}\varphi^{2\beta})^{\alpha_1}\, (\varrho^\gamma
\varphi^{-\beta})^{\alpha_2}(\varphi^{-2\beta})^{\alpha_3}\,
\varrho^{\alpha_4}.
$$
Applying the H\"{o}lder inequality and recalling \eqref{anna3} and
\eqref{inna5c} we arrive at
$$
\|\varrho\mathbf u\|_{L^1(\Omega)}\leq A^{\alpha_1} \,
B^{\alpha_2}\|\varphi^{-2\beta}\|_{L^1(\Omega)}^{\alpha_3}M^{\alpha_4}\leq
c_e A^{\frac{1}{2(2-\theta)}}\,
    B^{\frac{1}{(2-\theta)(2\gamma-1)}}
$$
Inserting  this inequalities into \eqref{anna16} we obtain
\eqref{anna15a}. In order to prove \eqref{anna15b}, notice that
\begin{equation*}
    (\varrho|\mathbf u|^{2}\varphi^{1-\beta})^s=
    (\varrho|\mathbf
    u|^{2(2-\theta)}\varphi^{2\beta})^{\beta_1}\,
    (\varrho^{\gamma}\varphi^{-\beta})^{\beta_2}\,
    (\varphi^{-2\beta})^{\beta_3}\,(\varrho)^{\beta_4}\varphi^\varkappa,
\end{equation*}
where
\begin{equation*}\begin{split}
    \beta_1=\frac{s}{2-\theta}, \quad
    \beta_2=\frac{s\theta}{(2-\theta)(2\gamma-1)}, \quad \beta_3=
    1-s\big(1-\frac{\theta(\gamma-1)}{(2-\theta)(2\gamma-1)}\big),
    \\
    \beta_4=s\frac{(1-\theta)(2\gamma-1)-\gamma\theta}{(2-\theta)(2\gamma-1)},
\quad
    \varkappa=s(1-\beta)+\beta\beta_2+2\beta \beta_3-2\beta \beta_1.
\end{split}\end{equation*} It follows from \eqref{anna100}
that $\beta_i$ are positive and
\begin{equation*}
    \beta_1+\beta_2+\beta_3+\beta_4=1.
    \end{equation*}
Moreover, we have
$$
2\beta_3+\beta_2-2\beta_1=2-2s+\frac{2s\theta(\gamma-1)}{(2-\theta)(2\gamma-1)}
+\frac{s\theta}{(2-\theta)(2\gamma-1)}
-\frac{2s(2\gamma-1)}{(2-\theta)(2\gamma-1)} =2-3s.
$$
It follows that
$$ \varkappa=s+\beta(2-4s)=\frac{1}{3-8\theta^2}\big(2\theta^2+80\theta^4)>0. $$
 In particular, we have
\begin{equation*}
    (\varrho|\mathbf u|^{2}\varphi^{1-\beta})^s\leq c_e
    (\varrho|\mathbf
    u|^{2(2-\theta)}\varphi^{2\beta})^{\beta_1}
    (\varrho^{\gamma}\varphi^{-\beta})^{\beta_2}
    (\varphi^{-2\beta})^{\beta_3}(\varrho)^{\beta_4}.
\end{equation*}
Applying the H\"{o}lder inequality we obtain
\begin{equation*}
\|\varrho|\mathbf u|^2\varphi^{1-\beta}\|_{L^s(\Omega)}\leq c_e
    A^{\frac{1}{2-\theta}}\,
    B^{\frac{\theta}{(2-\theta)(2\gamma-1)}}M^{\beta_4/s}\Big(\int_\Omega
    \varphi^{-2\beta}\, dx \Big)^{\beta_3/s}\leq c_e A^{\frac{1}{2-\theta}}\,
    B^{\frac{\theta}{(2-\theta)(2\gamma-1)}},
\end{equation*}
which yields \eqref{anna15b}.
\end{proof}

The following lemma gives  the estimates of the pressure function
and the energy density in terms of the quantity $A$.

\begin{lem}\label{lena13}
Let  a solution
  $(\mathbf u, \varrho)\in W^{1,2}_0(\Omega)\times
L^8(\Omega)$ to problem \eqref{inna5} be given by Lemma \ref{anna5}.
Furthermore assume that  $\theta$, $\beta$, and $s$ are given by
\eqref{anna100}. Then
\begin{gather}\label{lena15}
 \|p_\epsilon(\varrho)\varphi^{1-\beta}\|_{L^{s}(\Omega)}+
 \|\varphi^{-\beta}p_\epsilon(\varrho)\|_{L^1(\Omega)}+\\\nonumber
 \|\varphi^{-\beta}\varrho (\nabla\varphi\cdot \mathbf
 u)^2\|_{L^1(\Omega)}\leq c_e(1+A^{(1+\theta)/2}),\\
 \label{lena16}
 \|\mathbf u\|_{W^{1,2}_0(\Omega)}\leq c_e(1+A^{(1-2\theta)/4}),\\
 \label{lena17}
\|\varphi^{1-\beta}\varrho |\mathbf u|^2\|_{L^s(\Omega)}\leq c_e
(1+A^{(1+\theta)/2}).
\end{gather}

\end{lem}

\begin{proof}
Combining estimates \eqref{lena2} and \eqref{lena5} we obtain
\begin{equation}\label{lena16a}\begin{split}
     \|p_\epsilon(\varrho)\varphi^{1-\beta}\|_{L^{s}(\Omega)}+
 \|\varphi^{-\beta}p_\epsilon(\varrho)\|_{L^1(\Omega)}+
 \|\varphi^{-\beta}\varrho (\nabla\varphi\cdot \mathbf
 u)^2\|_{L^1(\Omega)}\leq \\
 c_e(1+\|p_\epsilon(\varrho)\varphi^{1-\beta}\|_{L^1(\Omega)}+\|\varrho|\mathbf
 u|^2\varphi^{1-\beta}\|_{L^s(\Omega)}+\|\mathbf
 u\|_{W^{1,2}_0(\Omega)}).
\end{split}\end{equation}
Let us estimate
$\|p_\epsilon(\varrho)\varphi^{1-\beta}\|_{L^1(\Omega)}$. Recall
that for every integrable function $g$, for every $1\leq \sigma\leq
\tau\leq \infty$, and for every $\upsilon \in(0,1)$, we have
\begin{equation}\label{lena17a}
    \|g\|_{L^{r}(\Omega)}\leq
    \|g\|_{L^\sigma(\Omega)}^{\upsilon \sigma/r}
    \|g\|_{L^\tau(\Omega)}^{(1-\upsilon )\tau/r} \text{~~where~~}
    r=\upsilon \sigma+(1-\upsilon) \tau.
\end{equation}
Now set
$$
g=(p_\varepsilon(\varrho)\varphi^{1-\beta})^{1/4}, \quad r=4,\quad
\sigma=1, \quad \tau=4s, \quad \upsilon=\frac{4s-4}{4s-1}.
$$
Obviously we have
$$
\|g\|_{L^{r}(\Omega)}=\|p_\epsilon(\varrho)\varphi^{1-\beta}\|_{L^1(\Omega)}^{1/4},
\quad
\|g\|_{L^{\tau}(\Omega)}=\|p_\epsilon(\varrho)\varphi^{1-\beta}\|_{L^s(\Omega)}
^{1/4}.
$$
On the other hand, relations \eqref{inna4}, \eqref{inna5c}, and
\eqref{inna3} imply
$$
\|g\|_{L^1(\Omega)}=\|p_\varepsilon^{1/4}\varphi^{(1-\beta)/4}\|_{L^1(\Omega)}
\leq c_e\|\varrho\|_{L^1(\Omega)}\leq c_e.
$$
Inserting  these estimates into  \eqref{lena17a} we arrive at
\begin{equation*}
    \|p_\epsilon(\varrho)\varphi^{1-\beta}\|_{L^1(\Omega)}\leq
    \|p_\epsilon(\varrho)\varphi^{1-\beta}\|_{L^s(\Omega)}^
    {(1-\upsilon)\tau/r}
\end{equation*}
Substituting this inequality into \eqref{lena16a} and noting that
$(1-\upsilon)\tau/r<1$ we get
\begin{equation*}\begin{split}
     \|p_\epsilon(\varrho)\varphi^{1-\beta}\|_{L^{s}(\Omega)}+
 \|\varphi^{-\beta}p_\epsilon(\varrho)\|_{L^1(\Omega)}+
 \|\varphi^{-\beta}\varrho (\nabla\varphi\cdot \mathbf
 u)^2\|_{L^1(\Omega)}\leq \\
 c_e(1+\|\varrho|\mathbf
 u|^2\varphi^{1-\beta}\|_{L^s(\Omega)}+\|\mathbf
 u\|_{W^{1,2}_0(\Omega)}).
\end{split}\end{equation*}
From this and \eqref{anna15} we obtain
\begin{equation}\label{lena18}\begin{split}
\|p_\epsilon(\varrho)\varphi^{1-\beta}\|_{L^{s}(\Omega)}+
 \|\varphi^{-\beta}p_\epsilon(\varrho)\|_{L^1(\Omega)}+
 \|\varphi^{-\beta}\varrho (\nabla\varphi\cdot \mathbf
 u)^2\|_{L^1(\Omega)}\leq \\
 c_e(1+A^{\frac{1}{4(2-\theta)}}\,
    B^{\frac{1}{2(2-\theta)(2\gamma-1)}}+
    A^{\frac{1}{2-\theta}}\,
    B^{\frac{\theta}{(2-\theta)(2\gamma-1)}}).
\end{split}\end{equation}
Next, it follows from the Young inequality  that for every
$\delta>0$,
\begin{equation*}
A^{\frac{1}{4(2-\theta)}}\,
    B^{\frac{1}{2(2-\theta)(2\gamma-1)}}\leq \delta
    B+c(\delta)A^{\kappa_1},
\end{equation*}
where
$$
\kappa_1=\frac{1}{2}\frac{(2\gamma-1)}{
2(2-\theta)(2\gamma-1)-1}=\frac{1}{2}\frac{1+8\theta}{2(2-\theta)(1+8\theta)-
(1-8\theta)}<1/2.
$$
Similarly, we have
\begin{equation*}
 A^{\frac{1}{2-\theta}}\,
    B^{\frac{\theta}{(2-\theta)(2\gamma-1)}})\leq \delta
    B+c(\delta)A^{\kappa_2},
\end{equation*}
where
$$
\kappa_2=\frac{2\gamma-1}{(2-\theta)(2\gamma-1)-\theta}=\frac{1}{2}+
\frac{\theta}{2}\frac{1}{1+7\theta}<\frac{1}{2}+ \frac{\theta}{2}.
$$
 We
thus get
\begin{equation*}
A^{\frac{1}{4(2-\theta)}}\,
    B^{\frac{1}{2(2-\theta)(2\gamma-1)}}\leq \delta
    B+c(\delta)A^{1/2},
\quad A^{\frac{1}{2-\theta}}\,
    B^{\frac{\theta}{(2-\theta)(2\gamma-1)}}\leq \delta
    B+c(\delta)A^{1/2+\theta/2}.
\end{equation*}
Inserting these inequalities into \eqref{lena18} we obtain
\begin{equation*}\begin{split}
\|p_\epsilon(\varrho)\varphi^{1-\beta}\|_{L^{s}(\Omega)}+
 \|\varphi^{-\beta}p_\epsilon(\varrho)\|_{L^1(\Omega)}+
 \|\varphi^{-\beta}\varrho (\nabla\varphi\cdot \mathbf
 u)^2\|_{L^1(\Omega)}\leq \\c\delta B+
 c(\delta)(1+A^{1/2+\theta/2}).
\end{split}\end{equation*}
Noting that
\begin{equation}\label{lena19}
B=\|\varphi^{-\beta}\varrho^\gamma\|_{L^1(\Omega)}\leq
c\|\varphi^{-\beta}p_\epsilon(\varrho)\|_{L^1(\Omega)},
\end{equation}
we arrive at
\begin{equation*}\begin{split}
\|p_\epsilon(\varrho)\varphi^{1-\beta}\|_{L^{s}(\Omega)}+(1-
c\delta)
 \|\varphi^{-\beta}p_\epsilon(\varrho)\|_{L^1(\Omega)}+
 \|\varphi^{-\beta}\varrho (\nabla\varphi\cdot \mathbf
 u)^2\|_{L^1(\Omega)}\leq \\
 c(\delta)(1+A^{1/2+\theta/2}).
\end{split}\end{equation*}
Choosing $\delta>0$ sufficiently small we obtain the desired
estimate \eqref{lena15}.

Now our task is  to derive  estimates \eqref{lena16} and
\eqref{lena17}. We begin with the observations that inequalities
\eqref{lena15} and \eqref{lena19} imply
\begin{equation}\label{lena20}
    B
    \leq c_e(1+A^{1/2+\theta/2}).
\end{equation}
 Inserting this estimate into \eqref{anna15a}
we obtain
\begin{equation}\label{lena21}
\|\mathbf u\|_{W^{1,2}_0(\Omega)}\leq c_e
    A^{\frac{1}{4(2-\theta)}}(1+A^{1/2+\theta/2})^{\frac{1}{2(2-\theta)(2\gamma-1)}}\,
        \leq c_e(1+A^{\kappa_3}),
\end{equation}
where
$$
\kappa_3=\frac{1}{4(2-\theta)}+
\frac{1+\theta}{4}\frac{1}{(2-\theta)(2\gamma-1)}=\frac{1}{4(2-\theta)}
\frac{2\gamma+\theta}{2\gamma-1}\leq \frac{1}{4(2-\theta)}
\frac{2+\theta}{1+8\theta}.
$$
Here we use the relations $2-\gamma^{-1}=1+8\theta$ and
$\gamma^{-1}< 1$. It follows that
$$
\frac{1-2\theta}{4}-\kappa_3\geq  \frac{1}{4(2-\theta)} h(\theta),
\text{~~where~~} h(\theta)=
(2-\theta)(1-2\theta)-\frac{2+\theta}{1+8\theta}.
$$
It is easily seen that
$$
(1+8\theta)h(\theta)=2\theta(5-19\theta+8\theta^2)>0,
$$
since $\theta\in (0,1/8)$. Hence $\kappa_3\leq 1/4-\theta/2$, which
along with \eqref{lena21} yields \eqref{lena16}.
 It remains to prove \eqref{lena17}.
Combining estimates \eqref{anna15b} and \eqref{lena20} we obtain
\begin{equation}\label{lena22}
    \|\varrho|\mathbf u|^2\varphi^{1-\beta}\|_{L^q(\Omega)}\leq c_e
    A^{\frac{1}{2-\theta}}\,
    B^{\frac{\theta}{(2-\theta)(2\gamma-1)}}\leq c_e(1+A^{\kappa_4}),
\end{equation}
where
$$
\kappa_4=\frac{1}{2-\theta}+
\frac{1+\theta}{2}\frac{\theta}{(2-\theta)(2\gamma-1)}.
$$
It is easy to check that
\begin{equation*}\begin{split}
\frac{1+\theta}{2}-\kappa_4=\frac{\theta(1+\theta)}{2(2-\theta)}\big(
\frac{1-\theta}{1+\theta}-\frac{1}{2\gamma-1}\big)=
\frac{\theta(1+\theta)}{2(2-\theta)}\big(
\frac{1-\theta}{1+\theta}-\frac{1-8\theta}{1+8\theta}\big)>0.
\end{split}\end{equation*}
Hence $\kappa_4\leq (1+\theta)/2$, which along with \eqref{lena22}
yields \eqref{lena17}.
\end{proof}

 The following lemma gives the  weighted estimates for the energy
density in terms of $A$.
\begin{lem}\label{sura2}
Let  a solution
  $(\mathbf u, \varrho)\in W^{1,2}_0(\Omega)\times
L^8(\Omega)$ to problem \eqref{inna5} be given by Lemma \ref{anna5}.
Then for every $\alpha\in (0,1)$ and $x_0\in\Omega$, we have
\begin{equation}\label{sura3}
    \int_\Omega \big(p_\epsilon(\varrho)+\varrho|\mathbf
    u|^2\big)(x)\, \varphi(x)^{3/2-\beta}\, |x-x_0|^{-\alpha}\, dx
    \leq c(1+A^{(1+\theta)/2}),
\end{equation}
where $c$ depends only on $c_e$ and $\alpha$.

\end{lem}

\begin{proof}
Fix an arbitrary $\alpha\in (0,1)$ and $x_0\in \Omega$. Introduce
the vector field
\begin{equation}\label{sura4}
    \boldsymbol\xi(x)=\varphi(x)^{3/2-\beta}|x-x_0|^{-\alpha}
    (x-x_0).
\end{equation}
It is easily seen that
\begin{equation*}\begin{split}
\nabla\boldsymbol\xi=\frac{\varphi(x)^{3/2-\beta}}{|x-x_0|^{\alpha}}
\Big(\mathbf I-\frac{\alpha}{|x-x_0|^2} (x-x_0)\otimes (x-x_0) +
\frac{3-2\beta}{2\varphi}\,(x-x_0)\otimes \nabla\varphi\Big).
\end{split}\end{equation*}
Recall that  $\beta\in (0,1/2)$. From this
 we conclude that
$$
|\nabla\boldsymbol\xi(x)|\leq c |x-x_0|^{-\alpha}.
$$
Hence the vector field $\boldsymbol\xi$ belongs to the class
$W^{1,r}_0(\Omega)$ for all $r\in [1,3/\alpha)$. In particular, we
have
$$
\|\boldsymbol\xi\|_{W^{1,2}_0(\Omega)}\leq c, \quad
\|\boldsymbol\xi\|_{L^\infty(\Omega)}\leq c.
$$
Substituting $\boldsymbol\xi $ into integral identity \eqref{anna7}
leads to
\begin{equation*}\begin{split}
     \int_\Omega\big(\varrho\mathbf u\otimes\mathbf u:\nabla\boldsymbol
    \xi+p_\epsilon(\varrho) \text{div~}\boldsymbol\xi\big)dx=
    \int_\Omega
    \big(\mathbb S(\mathbf u):\nabla\boldsymbol\xi+
    \varrho\mathbf f\cdot \boldsymbol\xi\big)\,
    dx\leq\\
    c \|\mathbf
    u\|_{W^{1,2}_0(\Omega)}\|\boldsymbol\xi\|_{W^{1,2}_0(\Omega)}+c
    M\|\boldsymbol\xi\|_{L^\infty(\Omega)},
\end{split}\end{equation*}
which yields the estimate
\begin{equation}\label{sura5}\begin{split}
 \int_\Omega\big(\varrho\mathbf u\otimes\mathbf u:\nabla\boldsymbol
    \xi+p_\epsilon(\varrho) \text{div~}\boldsymbol\xi\big)dx\leq
    c(1+\|\mathbf u\|_{W^{1,2}_0(\Omega)}).
\end{split}\end{equation}
We have
\begin{equation}\begin{split}\label{sura5e}
\text{div~}\boldsymbol\xi=\frac{(3-\alpha)\varphi^{3/2-\beta}}{|x-x_0|^\alpha}
-\frac{(3-2\beta)\varphi^{1/2-\beta}}{2|x-x_0|^{\alpha}}\,
(x-x_0)\cdot
\nabla\varphi\geq\\
\frac{(3-\alpha)\varphi^{3/2-\beta}}{|x-x_0|^\alpha}-c,
\end{split}\end{equation}
and
\begin{equation}\begin{split}\label{sura5ee}
\varrho\mathbf u\otimes\mathbf u:\nabla\boldsymbol
    \xi=\frac{\varphi^{3/2-\beta}\varrho}{|x-x_0|^\alpha}\Big(|\mathbf
    u|^2-\frac{\alpha}{|x-x_0|^2}\big(\mathbf
    u\cdot(x-x_0)\big)^2\Big)+\\
\varrho\frac{(3-2\beta)\varphi^{1/2-\beta}}{2|x-x_0|^{\alpha}}\,
\big((x-x_0)\cdot
\mathbf u)\big(\nabla\varphi\cdot\mathbf u\big)\geq\\
\frac{(1-\alpha)\varphi^{3/2-\beta}}{|x-x_0|^\alpha}\varrho|\mathbf
    u|^2-c\varphi^{1/2-\beta}\varrho|\mathbf u||\big(\nabla\varphi\cdot\mathbf
    u\big)|\geq\\
    \frac{(1-\alpha)\varphi^{3/2-\beta}}{|x-x_0|^\alpha}\varrho|\mathbf
    u|^2-c\varphi^{1-\beta}\varrho|\mathbf u|^2-c
    \varphi^{-\beta}\varrho
\big(\nabla\varphi\cdot\mathbf
    u\big)^2.
\end{split}\end{equation}
Inserting \eqref{sura5e} and \eqref{sura5ee} into  \eqref{sura5} we
arrive at the inequality
\begin{multline*}
 \int_\Omega\frac{\varphi^{3/2-\beta}(p_\epsilon +\varrho|\mathbf
    u|^2)}{|x-x_0|^\alpha}\, dx\leq\\ c\Big(1+\int_\Omega
    p_\epsilon\varphi^{-\beta}\, dx +
\int_\Omega\varphi^{1-\beta}\varrho|\mathbf u|^2\, dx+\int_\Omega
\varphi^{-\beta}\varrho \big(\nabla\varphi\cdot\mathbf
    u\big)^2\, dx+\|\mathbf u\|_{W^{1,2}_0(\Omega)}\Big).
\end{multline*}
From this and \eqref{lena15}-\eqref{lena17} we finally obtain
\begin{equation*}
\int_\Omega\frac{\varphi^{3/2-\beta}(p_\epsilon +\varrho|\mathbf
    u|^2)}{|x-x_0|^\alpha}\, dx\leq
    c(1+A^{(1+\theta)/2}+A^{(1-2\theta)/4})\leq
    c(1+A^{(1+\theta)/2}).
\end{equation*}

\end{proof}
\section{Estimates for $\mathbf u$ and $p_\epsilon$}\label{kira}
\renewcommand{\theequation}{5.\arabic{equation}}
\setcounter{equation}{0}

In this section we establish the a priori estimates for$\|\mathbf
u\|_{W^{1,2}_0(\Omega)}$ and the pressure function. The result is
given by the following
\begin{prop}\label{kuka200}
Let  a solution
  $(\mathbf u, \varrho)\in W^{1,2}_0(\Omega)\times
L^8(\Omega)$ to problem \eqref{inna5} be given by Lemma \ref{anna5}
and   $q$ be given by \eqref{anna100}. Then
\begin{equation}\label{kira101}
\|\mathbf u\|_{W^{1,2}_0(\Omega)}+\|p_\epsilon\|_{L^q(\Omega)}\leq
c_e,
\end{equation}
where the constant $c_e$ is specified by Remark \ref{anna200}.
\end{prop}
\begin{proof} The proof is based on the following  technical lemmas.

\begin{lem}\label{kira1}
Let  a solution
  $(\mathbf u, \varrho)\in W^{1,2}_0(\Omega)\times
L^8(\Omega)$ to problem \eqref{inna5} be given by Lemma \ref{anna5}
and  $(\theta,\beta)$    be given by  \eqref{anna100}. Then for
every $x_0\in\Omega$, we have
\begin{equation}\label{sura6}
    \int_\Omega \frac{\varrho(x)|\mathbf
    u|^{2(1-\theta)}(x)\, \varphi(x)^{2\beta}}{ |x-x_0|}\, dx\leq
 c_e(1+A^{(1+\theta)/2}).
\end{equation}

\end{lem}
\begin{proof}
Formulae \eqref{anna100} imply
$$
2\beta=\frac{\theta}{\gamma}\Big(\frac{3}{2}-\beta\Big)
+(1-\theta)\Big(\frac{3}{2}-\beta\Big).
$$
Next we set
$$
\alpha=(1-16\theta)/(1-8\theta)\in (0,1).
$$
 We have
\begin{equation*}
\frac{\varrho|\mathbf
    u|^{2(1-\theta)}\, \varphi^{2\beta}}{ |x-x_0|}=\Big(
    \frac{\varrho^\gamma\varphi^{3/2-\beta}}{|x-x_0|^\alpha}\Big)^{\theta/\gamma}
    \,\,\Big(\frac{\varrho|\mathbf
    u|^2\varphi^{3/2-\beta}}{|x-x_0|^\alpha}\Big)^{1-\theta}\Big(\frac{1}{|x-x_0|^2}
    \Big)^{\theta-\theta/\gamma}.
\end{equation*}
Applying the Young inequality and noting that $\varrho^\gamma\leq c
p_\epsilon(\varrho)$ we arrive at
\begin{equation*}
\frac{\varrho|\mathbf
    u|^{2(1-\theta)}\, \varphi^{2\beta}}{ |x-x_0|}\,\leq\, c
 \frac{p_\epsilon\,\varphi^{3/2-\beta}}{|x-x_0|^\alpha}\,+\,\frac{\varrho|\mathbf
    u|^2\,\varphi^{3/2-\beta}}{|x-x_0|^\alpha}\,+\,\frac{1}{|x-x_0|^2}
\end{equation*}
Integrating both sides of this inequality over $\Omega$ and using
estimate \eqref{sura3}, we obtain \eqref{sura6}.
\end{proof}

\begin{lem}\label{sura7}
Let  a solution
  $(\mathbf u, \varrho)\in W^{1,2}_0(\Omega)\times
L^8(\Omega)$ to problem \eqref{inna5} be given by Lemma \ref{anna5}.
 Then
\begin{equation}\label{sura8}
   A\leq c_e.
\end{equation}
\end{lem}
\begin{proof}
Estimate \eqref{sura6} and  Lemma \ref{sura7e} imply
\begin{equation}\label{sura10}\begin{split}
    A\equiv \int_\Omega |\mathbf u|^2(\varrho |\mathbf
    u|^{2(1-\theta)}\varphi^{2\beta})\,dx \leq\\
c\|\mathbf
u\|_{W^{1,2}_0(\Omega)}^2\sup\limits_{x_0\in\Omega}\int_\Omega
\varrho |\mathbf
    u|^{2(1-\theta)}\varphi^{2\beta}|x-x_0|^{-1}\, dx\leq
    c_e(1+A^{(1+\theta)/2})\|\mathbf
u\|_{W^{1,2}_0(\Omega)}^2 \end{split}\end{equation} From this, the
inequalities $0<\theta<1/8$,  and \eqref{lena16} we get the
inequality
$$
A\leq c_e (1+A^{1/2-\theta})(1+A^{(1+\theta)/2})\leq c_e
(1+A^{1-\theta/2}),
$$
which obviously yields \eqref{sura8}.
\end{proof}
Let us turn to the proof of Proposition \ref{kuka200}. Estimate
\eqref{kira101} for $\|\mathbf u\|_{W^{1,2}_0(\Omega)}$ obviously
follows from Lemmas \ref{lena13} and \ref{sura7}.  It remains to
estimate $\|p_\epsilon\|_{L^q(\Omega)}$. Recall  formula
\eqref{anna100} for $q$. Next, the Young inequality implies
\begin{equation*}\begin{split}
p_\epsilon^q=\big\{\,\big(p_\epsilon\varphi^{1-\beta}\big)^s\,\big\}^{\frac{\beta}{
\beta+(1-\beta)s}}\,\big\{\,p_\epsilon\varphi^{-\beta}\,\big\}^{\frac{(1-\beta)s}{
\beta+(1-\beta)s}}\leq (p_\epsilon\varphi^{1-\beta})^s+
p_\epsilon\varphi^{-\beta}.
\end{split}\end{equation*}
Integrating both sides of this relation over $\Omega$ and applying
estimates \eqref{lena15} and \eqref{sura8}, we obtain the desired
estimate
\begin{equation*}\begin{split}
\int_\Omega p_\epsilon^q\, dx\leq\int_\Omega
\big(p_\epsilon\varphi^{1-\beta}\big)^s\, dx+ \int_\Omega
p_\epsilon\varphi^{-\beta}\, dx\leq c_e (1+A^{(1+\theta)/2})\leq
c_e.
\end{split}\end{equation*}

\end{proof}

\section{Kinetic energy estimate.  Proof of Theorem \ref{inna11}}\label{olga}
\renewcommand{\theequation}{6.\arabic{equation}}
\setcounter{equation}{0}

In this section we establish a priori estimate of the kinetic energy
density,  and by doing so  we complete the proof of Theorem
\ref{inna11}.

\begin{prop}
\label{olga31} Let  a solution
  $(\mathbf u, \varrho)\in W^{1,2}_0(\Omega)\times
L^8(\Omega)$ be a weak  renormalized solution to problem
\eqref{inna5} and $s$ be given by \eqref{anna100}. Then
\begin{equation}\label{olga32}\begin{split}
\|\varrho|\mathbf u|^2\|_{L^s(\Omega)}\leq c_e.
\end{split}\end{equation}
\end{prop}
We  divide  the proof into a sequence of lemmas. Let us consider the
function $\varphi\in C^2(\Omega)$ given by  Definition
\ref{anna105}. In view of this definition,  there is $t>0$ such that
$\varphi(x)$ equals the signed distance function in the annulus
$$
A_{2t}=\{x\in \mathbb R^3:\,\, \text{dist~}(x,\partial\Omega)<2t\}.
$$
Introduce the vector field
\begin{equation*}
    \mathbf n(x)=\nabla\varphi(x), \quad  \mathbf n\in C^1(\bar
    A_{2t}), \quad |\mathbf n(x)|=1.
\end{equation*}
Fix an arbitrary $\alpha\in (0,1)$ and $x_0\in A_t$. Define the
vector field
\begin{equation}\label{olga1}
\boldsymbol\xi(x)=\Big\{\frac{\varphi(x)-\varphi(x_0)}{\Delta_-(x,x_0)^\alpha}+
\frac{\varphi(x)+\varphi(x_0)}{\Delta_+(x,x_0)^\alpha} \Big\}\,
\mathbf n(x),
\end{equation}
where
\begin{equation*}
\Delta_\pm(x,x_0)\,=\,|\varphi(x)\pm\,\varphi(x_0)|+|x-x_0|^2.
\end{equation*}
The following two lemmas, which proofs are given in the appendix,
constitute the basic properties of $\boldsymbol\xi$.
\begin{lem}\label{olga2}
There is a constant $c$, depending only on $\alpha$ and $\Omega$,
such that for every $x,x_0\in A_t$ and for every $\mathbf u\in
\mathbb R^3$,
\begin{subequations}\label{olga3}
\begin{gather}\label{olga3a}
|\boldsymbol\xi(x)|\leq c, \quad |\nabla\boldsymbol\xi(x)|\leq c
\Big(\frac{1}{\Delta_-(x,x_0)^\alpha}+\frac{1}{\Delta_+(x,x_0)^\alpha}+1\Big),\\
\label{olga3c} \frac{\partial\xi_i}{\partial x_j}(x) u_iu_j \geq
\frac{1-\alpha}{2}
\Big(\frac{1}{\Delta_-(x,x_0)^\alpha}+\frac{1}{\Delta_+(x,x_0)^\alpha}\Big)|\mathbf
u\cdot\mathbf n(x)|^2-c |\mathbf u|^2,\\
 \label{olga3b}
\text{\rm~div~}\boldsymbol\xi\geq \frac{1-\alpha}{2}
\Big(\frac{1}{\Delta_-(x,x_0)^\alpha}+\frac{1}{\Delta_+(x,x_0)^\alpha}\Big)-c.
\end{gather}
\end{subequations}
\end{lem}
\begin{proof} The proof is in the Appendix.
\end{proof}
\begin{lem}\label{olga10}
Let $\Omega_t=\Omega\cap A_t$. Then there is a constant $c$,
depending only on $\alpha$ and $\Omega$ such that
\begin{equation}\label{olga11}
    \|\nabla\boldsymbol\xi\|_{L^2(\Omega_t)}\leq c\text{~~for every~~}
    x_0\in \Omega_t.
\end{equation}
\end{lem}
\begin{proof} The proof is in the Appendix.
\end{proof}

The next lemma gives the weighted pressure estimate near $\partial
\Omega$.
\begin{lem}\label{olga16}
Let  a solution
  $(\mathbf u, \varrho)\in W^{1,2}_0(\Omega)\times
L^8(\Omega)$ to problem \eqref{inna5} be given by Lemma \ref{anna5}.
Let  $\alpha\in (0,1)$ and  $x_0\in \Omega$. Furthermore, assume
that $\zeta\in C^\infty(\bar\Omega)$ satisfies the conditions
\begin{equation}\label{olga17}
    \zeta \geq 0\text{~~in~~}\Omega, \quad
    \zeta=0\text{~~in~~}\Omega\setminus \Omega_{t/2}.
\end{equation}
Then
\begin{equation}\label{olga18}\begin{split}
    \int_\Omega\frac{\zeta p_\epsilon(\varrho)(x)}{|x-x_0|^\alpha}\, dx\leq c
\big(\|p_\epsilon(\varrho)\|_{L^1(\Omega)}+\|\varrho|\mathbf
u|^2\|_{L^1(\Omega)}+\|\mathbf u\|_{W^{1,2}_0(\Omega)}+1\big),
\end{split}\end{equation}
where $c$ depends on  $c_e$, $\alpha$, and
 $\zeta$.
\end{lem}
\begin{proof} We first consider the case of  $x_0\in \Omega_t$. In this case,
 the vector field $\boldsymbol\xi$
  meets all requirements of
Lemmas \ref{olga2} and \ref{olga10}. It follows from Lemma
\ref{olga10} that $\zeta\boldsymbol \xi \in W^{1,2}_0(\Omega)$.
Integral identity \eqref{inna6} with $\boldsymbol\xi$ replaced by
$\zeta\boldsymbol \xi $ implies
\begin{equation}\label{olga19}\begin{split}
    \int_\Omega\big(\zeta\varrho\mathbf u\otimes\mathbf u:\nabla\boldsymbol
    \xi+\zeta p_\epsilon(\varrho) \text{div~}\boldsymbol\xi\big)\, dx=\\
    \int_\Omega\big(\mathbb
    S(\mathbf u):\nabla(\zeta\boldsymbol\xi)-
    \varrho\mathbf f\cdot \zeta\boldsymbol\xi\big)\,
    dx- \int_\Omega\big(\varrho (\nabla\zeta\cdot\mathbf u)
    (\mathbf u\cdot\boldsymbol
    \xi)+ p_\epsilon(\varrho) \nabla\zeta\cdot\boldsymbol\xi\big)\,
    dx.
\end{split}\end{equation}
Notice that $|\boldsymbol\xi|\leq 4(\text{diam~}\Omega)^{1-\alpha}$.
From this and Lemma \ref{olga10} we conclude that
$\|\zeta\boldsymbol\xi\|_{W^{1,2}_0(\Omega)}\leq c$. It follows that
$$
\int_\Omega\big(\mathbb
    S(\mathbf u):\nabla(\zeta\boldsymbol\xi)-
    \varrho\mathbf f\cdot \zeta\boldsymbol\xi\big)\,
    dx\leq c
\big(\|\mathbf u\|_{W^{1,2}_0(\Omega)}+1\big)
$$
and
$$
\int_\Omega\big(\varrho (\nabla\zeta\cdot\mathbf u)
    (\mathbf u\cdot\boldsymbol
    \xi)+ p_\epsilon(\varrho) \nabla\zeta\cdot\boldsymbol\xi\big)\, dx
\leq c\big(\|p_\epsilon(\varrho)\|_{L^1(\Omega)}+\|\varrho|\mathbf
u|^2\|_{L^1(\Omega)}\big).
$$
Combining these results with \eqref{olga19} we arrive at the
estimate
\begin{equation}\label{olga20}\begin{split}
    \int_\Omega\big(\zeta\varrho\mathbf u\otimes\mathbf u:\nabla\boldsymbol
    \xi+\zeta p_\epsilon(\varrho) \text{div~}\boldsymbol\xi\big)\, dx\leq\\
 c\big(\|p_\epsilon(\varrho)\|_{L^1(\Omega)}+\|\varrho|\mathbf
u|^2\|_{L^1(\Omega)}+\|\mathbf u\|_{W^{1,2}_0(\Omega)}+1\big)
\text{~~for all~~}x_0\in A_t .
\end{split}\end{equation}
Next, Lemma \ref{olga2} implies that
$$
\varrho\mathbf u\otimes\mathbf u:\nabla\boldsymbol
    \xi=\varrho\frac{\partial\xi_i}{\partial x_j}u_iu_j\geq
\frac{1-\alpha}{2}
\Big(\frac{1}{\Delta_-(x,x_0)^\alpha}+\frac{1}{\Delta_+(x,x_0)^\alpha}\Big)
\varrho|\mathbf u\cdot\mathbf n(x)|^2-c \varrho|\mathbf u|^2
$$
in $A_t$. From this, \eqref{olga3b}, and \eqref{olga20} we obtain
\begin{equation}\label{olga21}\begin{split}
    \int_\Omega \zeta p_\epsilon(\varrho)
    \Big(\frac{1}{\Delta_-(x,x_0)^\alpha}+\frac{1}{\Delta_+(x,x_0)^\alpha}\Big)
    \, dx\leq\\
 c\big(\|p_\epsilon(\varrho)\|_{L^1(\Omega)}+\|\varrho|\mathbf
u|^2\|_{L^1(\Omega)}+\|\mathbf u\|_{W^{1,2}_0(\Omega)}+1\big)
\text{~~for all~~}x_0\in A_t.
\end{split}\end{equation}
 Next,
the inequality $|\varphi(x)-\varphi(x_0)|\leq |x-x_0|$ yields $
\Delta_-(x,x_0)\leq c|x-x_0|$, where $c>0$ depends only on $\Omega$.
Combining this result with \eqref{olga21} we arrive at the estimate
\begin{equation}\label{olga22}\begin{split}
    \int_\Omega \zeta p_\epsilon(\varrho)|x-x_0|^{-\alpha}\, dx
  \leq\\
 c\big(\|p_\epsilon(\varrho)\|_{L^1(\Omega)}+\|\varrho|\mathbf
u|^2\|_{L^1(\Omega)}+\|\mathbf u\|_{W^{1,2}_0(\Omega)}\big)
\text{~~for all~~}x_0\in A_t.
\end{split}\end{equation}
Let us consider the case of $x_0\in \Omega\setminus A_{t}$. Since
$\zeta$ vanishes in $\Omega\setminus A_{t/2}$, the inequality
$2|x-x_0|\geq t$ holds for all  $x\in \text{spt~}\zeta$ and $x_0\in
\Omega\setminus A_t$, and hence
\begin{equation*}\begin{split}
    \int_\Omega \zeta p_\epsilon(\varrho)|x-x_0|^{-\alpha}\, dx
  \leq
 c\|p_\epsilon(\varrho)\|_{L^1(\Omega)}
\text{~~for all~~}x_0\in \Omega\setminus A_t.
\end{split}\end{equation*}
Combining this inequality with \eqref{olga22} we obtain the desired
estimate \eqref{olga18}.

\end{proof}

\begin{lem}\label{olga23}
Let  a solution
  $(\mathbf u, \varrho)\in W^{1,2}_0(\Omega)\times
L^8(\Omega)$ to problem \eqref{inna5} be given by Lemma \ref{anna5}.
 Then for every nonnegative function $\eta\in C^\infty_0(\Omega)$
and every $x_0\in\Omega$,
\begin{equation}\label{olga24}\begin{split}
    \int_\Omega\frac{\eta p_\epsilon(\varrho)(x)}{|x-x_0|}\, dx\leq c
\big(\|p_\epsilon(\varrho)\|_{L^1(\Omega)}+\|\varrho|\mathbf
u|^2\|_{L^1(\Omega)}+\|\mathbf u\|_{W^{1,2}_0(\Omega)}+1\big),
\end{split}\end{equation}
where $c$ depends only on $\eta$ and $c_e$.
\end{lem}
\begin{proof}Fix an arbitrary $x_0\in \Omega$ and introduce the vector
field
$$
\boldsymbol\xi_{\rm int}(x)=|x-x_0|^{-1} (x-x_0).
$$
Obviously $|\nabla\boldsymbol\xi_{\rm int}(x)|\leq c |x-x_0|^{-1}$.
Hence $\eta\boldsymbol\xi_{\rm int}\in W^{1,2}_0(\Omega)$ and
\begin{equation}\label{olga25}
    \|\eta \boldsymbol\xi_{\rm int}\|_{W^{1,2}_0(\Omega)}\leq c.
\end{equation}
Integral identity \eqref{inna6} with $\boldsymbol\xi$ replaced by
$\eta\boldsymbol \xi_{\rm int} $ implies
\begin{equation*}\begin{split}
    \int_\Omega\big(\eta\varrho\mathbf u\otimes\mathbf u:\nabla\boldsymbol
    \xi_{\rm int}+\eta p_\epsilon(\varrho) \text{div~}\boldsymbol\xi_{\rm int}\big)\, dx=\\
    \int_\Omega\big(\mathbb
    S(\mathbf u):\nabla(\eta\boldsymbol\xi_{\rm int})-
    \varrho\mathbf f\cdot \eta\boldsymbol\xi_{\rm int}\big)\,
    dx- \int_\Omega\big(\varrho (\nabla\eta\cdot\mathbf u)
    (\mathbf u\cdot\boldsymbol
    \xi_{\rm int})+ p_\epsilon(\varrho)
    \nabla\eta\cdot\boldsymbol\xi_{\rm int}\big)\, dx
\end{split}\end{equation*}

 From this, \eqref{olga25} and the obvious equality $|
\boldsymbol\xi_{\rm int}|=1$ we obtain
\begin{equation}\label{olga27}\begin{split}
    \int_\Omega\big(\eta\varrho\mathbf u\otimes\mathbf u:\nabla\boldsymbol
    \xi_{\rm int}+\eta p_\epsilon(\varrho) \text{div~}\boldsymbol\xi_{\rm int}\big)\, dx\leq\\
 c\big(\|p_\epsilon(\varrho)\|_{L^1(\Omega)}+\|\varrho|\mathbf
u|^2\|_{L^1(\Omega)}+\|\mathbf u\|_{W^{1,2}_0(\Omega)}+1\big)
\end{split}\end{equation}
Straightforward calculations give
\begin{equation*}
\mathbf u\otimes\mathbf u:\nabla\boldsymbol
    \xi_{\rm int}=|x-x_0|^{-1}(\mathbb I-\boldsymbol
    \xi_{\rm int}\otimes\boldsymbol
    \xi_{\rm int})\mathbf u\cdot\mathbf u\geq 0,
\end{equation*}
and
\begin{equation*}
\text{div~}\boldsymbol\xi_{\rm int}=2|x-x_0|^{-1}.
\end{equation*}
Combining these results with \eqref{olga27} we obtain
\eqref{olga24}.
\end{proof}
\begin{lem}\label{olga28}
Let  a solution
  $(\mathbf u, \varrho)\in W^{1,2}_0(\Omega)\times
L^8(\Omega)$ to problem \eqref{inna5} be given by Lemma \ref{anna5}.
Let  $\alpha\in (0,1)$. Then for every  $x_0\in \Omega$,
\begin{equation}\label{olga29}\begin{split}
    \int_\Omega\frac{ p_\epsilon(\varrho)^\alpha}{|x-x_0|}\,dx\leq c
\big(\|\varrho|\mathbf u|^2\|_{L^1(\Omega)}+1\big),
\end{split}\end{equation}
where $c$ depends only  on  $c_e$ and $\alpha$.
\end{lem}
\begin{proof}
Choose a nonnegative function $\zeta\in C^\infty(\Omega)$ such that
$\zeta$ equals $1$ in a neighborhood of $\partial\Omega$ and $\zeta$
vanishes in $\Omega\setminus \Omega_{t/2}$. In particular, we have
$1-\zeta\in C^\infty_0(\Omega)$. Applying Lemmas \ref{olga16} and
\ref{olga23} we obtain
\begin{equation}\label{olga30}\begin{split}
\int_\Omega\frac{ p_\epsilon(\varrho)}{|x-x_0|^\alpha}\,dx\leq
\int_\Omega\frac{\zeta p_\epsilon(\varrho)}{|x-x_0|^\alpha}\, dx +
c \int_\Omega\frac{(1-\zeta) p_\epsilon(\varrho)}{|x-x_0|}\, dx\leq\\
c\big(\|p_\epsilon(\varrho)\|_{L^1(\Omega)}+\|\varrho|\mathbf
u|^2\|_{L^1(\Omega)}+\|\mathbf u\|_{W^{1,2}_0(\Omega)}+1\big).
\end{split}\end{equation}
On the other hand, the Young inequality implies
\begin{equation*}\begin{split}
\frac{ p_\epsilon(\varrho)^\alpha}{|x-x_0|}=\Big(\frac{
p_\epsilon(\varrho)}{|x-x_0|^{\alpha}}\Big)^{\alpha} \Big(\frac{
1}{|x-x_0|^{1+\alpha}}\Big)^{1-\alpha}\leq \frac{
cp_\epsilon(\varrho)}{|x-x_0|^{\alpha}}+\frac{
c}{|x-x_0|^{1+\alpha}}.
\end{split}\end{equation*}
Integrating both sides over $\Omega$ and noting that $1+\alpha\leq
2$ we obtain
\begin{equation*}\begin{split}
\int_\Omega\frac{ p_\epsilon(\varrho)^\alpha}{|x-x_0|}\, dx\leq
c\int_\Omega\frac{ p_\epsilon(\varrho)}{|x-x_0|^{\alpha}}\, dx +c.
\end{split}\end{equation*}
Combining this result with \eqref{olga30} we arrive at
\begin{equation*}
\int_\Omega\frac{ p_\epsilon(\varrho)^\alpha}{|x-x_0|}\,dx\leq
c\big(\|p_\epsilon(\varrho)\|_{L^1(\Omega)}+\|\varrho|\mathbf
u|^2\|_{L^1(\Omega)}+\|\mathbf u\|_{W^{1,2}_0(\Omega)}+1\big).
\end{equation*} It remains to note that
$$
\|p_\epsilon(\varrho)\|_{L^1(\Omega)}+\|\mathbf
u\|_{W^{1,2}_0(\Omega)}\leq c_e,
$$
and the lemma follows.
\end{proof}

We are now in a position to complete the proof of Proposition
\ref{olga31}. In view of \eqref{anna100} we have
$$
2\gamma^{-1}s-(3-s)=2(1-8\theta)(1+2\theta^2)-2+2\theta^2=
2\theta(3\theta+16\theta^2-8)<0
$$
since $\theta<1/8$. Set    $\alpha:=2\gamma^{-1}s/(3 -s)\in (0,1)$.
It follows from Lemma \ref{sura7e}, inequality \eqref{olga29}, and
estimate \eqref{kira101} that
\begin{equation*}
    \int_\Omega p_\epsilon(\varrho)^\alpha|\mathbf u|^2\,dx\leq c
\big(\|\varrho|\mathbf u|^2\|_{L^1(\Omega)}+1\big)\|\mathbf
u\|_{W^{1,2}_0(\Omega)}^2\leq  c \big(\|\varrho|\mathbf
u|^2\|_{L^1(\Omega)}+1\big).
\end{equation*}
On the other hand, we have
$$
\varrho^{2s/(3-s)}=  \varrho^{\alpha\gamma}\leq c
p_\epsilon(\varrho)^\alpha.
$$
We thus get
\begin{equation}\label{olga33}
     \int_\Omega \varrho^{2s/(3-s)}|\mathbf u|^2\,dx\leq c
     \big(\|\varrho|\mathbf
u|^2\|_{L^1(\Omega)}+1\big).
\end{equation}
Notice that $$\varrho^s|\mathbf u|^{2s}=\big(
\varrho^{2s/(3-s)}\big)^{(3-s)/2}\, \big(|u|^6\big)^{(s-1)/2}. $$
Applying the H\"{o}lder inequality and using \eqref{olga33} we
obtain
\begin{equation*}
    \int_\Omega \varrho^s|\mathbf u|^{2s}\leq
    \Big(\int_\Omega \varrho^{2s/(3-s)}|\mathbf u|^2\,dx\Big)^{(3-s)/2}
    \Big(\int_\Omega |u|^6\, dx\Big)^{(s-1)/2}.
\end{equation*}
It follows from the embedding theorem that
$$
\|\mathbf u\|_{L^6(\Omega)}\leq c\|\mathbf
u\|_{W^{1,2}_0(\Omega)}\leq c_e,
$$
which yields
\begin{equation}\label{olga34}
    \int_\Omega \varrho^s|\mathbf u|^{2s}\, dx\leq
   c \Big(\int_\Omega \varrho^{2s/(3-s)}|\mathbf u|^2\,dx\Big)^{(3-s)/2}.
\end{equation}
Inserting  \eqref{olga33} into \eqref{olga34} we arrive at the
inequality
\begin{equation}\label{olga35}
    \int_\Omega \varrho^s|\mathbf u|^{2s}\,dx\leq\big(\|\varrho|\mathbf
u|^2\|_{L^1(\Omega)}+1\big)^{(3-s)/2}.
\end{equation}
Notice that
$$
\|\varrho|\mathbf u|^2\|_{L^1(\Omega)}+1\leq c\Big(\int_\Omega
\varrho^s|\mathbf u|^{2s}\, dx\Big)^{1/s}+1\leq c\Big(\int_\Omega
\varrho^s|\mathbf u|^{2s}\, dx+1\Big)^{1/s}.
$$
Substituting this inequality into \eqref{olga35} we finally obtain
\begin{equation*}
    \int_\Omega \varrho^s|\mathbf u|^{2s}\,dx\leq
c\Big( \int_\Omega\varrho^s|\mathbf u|^{2s}\,
dx+1\Big)^{(3-s)/(2s)}.
\end{equation*}
It remains to note that $3-s<2s$ and the proposition follows.

{\bf Proof of Theorem \ref{inna11}.} It suffices to note that
estimate \eqref{inna10} is a straightforward consequence of
Propositions \ref{kuka200} and \ref{olga31}.

\appendix
\section{Proof of Lemmas \ref{olga2} and
\ref{olga10}}\label{appendix}
\renewcommand{\theequation}{7.\arabic{equation}}
\setcounter{equation}{0}

{\bf Proof of Lemma \ref{olga2}.}  The first estimate in
\eqref{olga3a} is obvious. In order to prove the second estimate
notice that
\begin{equation}\label{olga4}
 \frac{\partial\xi_i}{\partial x_j}(x)=P(x,x_0)n_i(x)n_j(x)+Q_j(x,
 x_0)n_i(x)+R(x,x_0)\frac{\partial n_i}{\partial x_j}(x),
\end{equation}
where
\begin{subequations}\label{olga5}
\begin{gather}\label{olga5a}
P(x,x_0)=
\frac{(1-\alpha)|\varphi(x)-\varphi(x_0)|+|x-x_0|^2}{\Delta_-(x,x_0)^{1+\alpha}}+
\frac{(1-\alpha)|\varphi(x)+\varphi(x_0)|+|x-x_0|^2}{\Delta_+(x,x_0)^{1+\alpha}},\\
\label{olga5b} Q_j(x,x_0) =-2\alpha
\Big(\frac{\varphi(x)-\varphi(x_0)}{\Delta_-(x,x_0)^{1+\alpha}}+
\frac{\varphi(x)+\varphi(x_0)}{\Delta_+(x,x_0)^{1+\alpha}}\Big)\,(x-x_0)_j,\\
\label{olga5c} R(x,x_0)=
\frac{\varphi(x)-\varphi(x_0)}{\Delta_-(x,x_0)^\alpha}\,+\,
\frac{\varphi(x)+\varphi(x_0)}{\Delta_+(x,x_0)^\alpha}.
\end{gather}
\end{subequations}
It follows that
\begin{equation}\label{olga6}
(1-\alpha)\Big(\frac{1}{\Delta_-^\alpha}+\frac{1}{\Delta_+^\alpha}\Big)
\leq P(x,x_0)\leq
\Big(\frac{1}{\Delta_-^\alpha}+\frac{1}{\Delta_+^\alpha}\Big).
\end{equation}
\begin{equation}\label{olga7}
    |R|\leq c.
\end{equation}
Next we have
\begin{multline*}
    |\varphi(x)\pm\varphi(x_0)||(x-x_0)_j|\leq c |\varphi(x)\pm\varphi(x_0)|^{1/2}
    |x-x_0|\leq\\ c|\varphi(x)\pm\varphi(x_0)|+c|x-x_0|^2\leq
    c\Delta_\pm(x,x_0).
\end{multline*}
From this and \eqref{olga5b} we obtain
\begin{equation*}
    |Q_j|\leq c\Big(\frac{1}{\Delta_-^\alpha}+
    \frac{1}{\Delta_+^\alpha}\Big)
\end{equation*}
Inserting this inequality and inequalities \eqref{olga6},
\eqref{olga7} into \eqref{olga4} we arrive at \eqref{olga3a}. Next,
\eqref{olga4} and \eqref{olga5} imply
\begin{equation}\label{olga8}
\frac{\partial\xi_i}{\partial x_j} u_iu_j =P(\mathbf u\cdot\mathbf
n)^2-(N_-+N_+)((x-x_0)\cdot\mathbf u)(\mathbf n\cdot\mathbf
u)+R\frac{\partial n_i}{\partial x_j}u_iu_j,
\end{equation}
where
$$
N_{\pm}(x, x_0)=2\alpha
\,\big(\varphi(x)\pm\varphi(x_0)\big)\Delta_{\pm}(x,x_0)^{-1-\alpha}.
$$
It follows from this and \eqref{olga6}, \eqref{olga7} that
\begin{equation}\label{olga9}\begin{split}
\frac{\partial\xi_i}{\partial x_j} u_iu_j \geq
(1-\alpha)\Big(\frac{1}{\Delta_-^\alpha}+
\frac{1}{\Delta_+^\alpha}\Big)(\mathbf n\cdot\mathbf u)^2-\\( |
N_-|+|N_+|)|x-x_0||\mathbf u||\mathbf n\cdot\mathbf u|-c |\mathbf
u|^2. \end{split}\end{equation} It follows from the expression for
$N_\pm$ and the Cauchy inequality that
\begin{equation*}\begin{split}
    ( |
N_-|+|N_+|)|x-x_0||\mathbf u||\mathbf n\cdot\mathbf
u|\leq\\
c\Big(\frac{|\varphi(x)-\varphi(x_0)||x-x_0|}{\Delta_-^{1+\alpha}}
+\frac{|\varphi(x)+\varphi(x_0)||x-x_0|}{\Delta_+^{1+\alpha}}\Big)
|\mathbf u||\mathbf n\cdot\mathbf u|\leq\\
\frac{1-\alpha}{2}\Big(\frac{|x-x_0|^2}{\Delta_-^{1+\alpha}}
+\frac{|x-x_0|^2}{\Delta_+^{1+\alpha}}\Big)(\mathbf n\cdot\mathbf
u)^2+\\c
\Big(\frac{|\varphi(x)-\varphi(x_0)|^2}{\Delta_-^{1+\alpha}}
+\frac{|\varphi(x)+\varphi(x_0)|^2}{\Delta_+^{1+\alpha}}\Big)|\mathbf
u|^2.
 \end{split}\end{equation*}
Notice that
$$
\frac{|x-x_0|^2}{\Delta_-(x,x_0)}\leq 1, \quad
\frac{|x-x_0|^2}{\Delta_+(x,x_0)}\leq 1
$$
and
$$
\frac{|\varphi(x)-\varphi(x_0)|^2}{\Delta_-(x,x_0)^{1+\alpha}}
+\frac{|\varphi(x)+\varphi(x_0)|^2}{\Delta_+(x,x_0)^{1+\alpha}}\leq
c.
$$
We thus get
\begin{equation*}\begin{split}
    ( |
N_-|+|N_+|)|x-x_0||\mathbf u||\mathbf n\cdot\mathbf u|\leq
\frac{1-\alpha}{2}\Big(\frac{1}{\Delta_-^{\alpha}}
+\frac{1}{\Delta_+^{\alpha}}\Big)(\mathbf n\cdot\mathbf
u)^2+c|\mathbf u|^2.
 \end{split}\end{equation*}
Inserting this inequality into \eqref{olga9} we arrive at the
desired inequality \eqref{olga3c}. It remains to prove
\eqref{olga3b}. To this end set $\mathbf u=\mathbf e_k$, where
$\mathbf e_k$ is a vector of the canonical basis in $\mathbf R^3$.
Substituting $\mathbf u$ into \eqref{olga3c} we obtain
\begin{equation*}
    \frac{\partial\xi_k}{\partial x_k}(x)  \geq
\frac{1-\alpha}{2}
\Big(\frac{1}{\Delta_-^\alpha}+\frac{1}{\Delta_+^\alpha}\Big)
n_k^2-c |\mathbf u|^2.
\end{equation*}
Summing both sides over $k$ and noting that $|\mathbf n|=1$ we
obtain \eqref{olga3b}. This completes the proof.

{\bf Proof of Lemma \ref{olga10}.} We begin with the observation
that $\Delta_-(x,x_0)\leq \Delta_+(x,x_0)$ for all $x,x_0\in
\Omega_t$. From this and \eqref{olga3a} we conclude that
$$
|\nabla\boldsymbol\xi(x)|^2\leq c
\Big(\frac{1}{\Delta_-(x,x_0)^{2\alpha}}+1\Big).
$$
Recall that $\Omega_t\subset A_t$. Hence it suffices to prove that
\begin{equation}\label{olga12}
    \int_{A_t}\Delta_-(x,x_0)^{-2\alpha}\, dx\leq c \text{~~for
    all~~} x_0\in A_t.
\end{equation}
To this end fix an arbitrary $x_0\in A_t$ and denote by $B_t$ the
ball $\{|x-x_0|<t\}$. We have
\begin{equation}\label{olga13}
    \int_{A_t}\Delta_-(x,x_0)^{-2\alpha}\, dx\leq
 \int_{B_t}\Delta_-(x,x_0)^{-2\alpha}\, dx+ \int_{A_t\setminus B_t}
 \Delta_-(x,x_0)^{-2\alpha}\, dx.
\end{equation}
It is easily seen that $\Delta_-(x,x_0)\geq t^2$ for all $x\in
A_t\setminus B_t$, which leads to the estimate
\begin{equation}\label{olga14}
\int_{A_t\setminus B_t}
 \Delta_-(x,x_0)^{-2\alpha}\, dx\leq
 ct^{-4\alpha}\text{~meas~}A_t\leq c.
\end{equation}
Recall that  $B_t\subset A_{2t}$ and that the function $\varphi$
belongs to the class $ C^2(\bar A_{2t}) $. From this and the Taylor
formula we obtain
$$
\varphi(x)-\varphi(x_0)=\mathbf n_0\cdot (x-x_0)+D(x,
x_0)\text{~~for~~} x\in B_t,
$$
where $\mathbf n_0=\mathbf n(x_0)=\nabla\varphi(x_0)$ and the
remainder admits the estimate
$$
|D(x,x_0)|\leq m|x-x_0|^2,\text{~~where~~} m=\sup\limits_{x\in
A_{2t}}|\nabla^2 \varphi(x)|.
$$
It follows that
\begin{equation*}\begin{split}
(m+1)\Delta_-(x,x_0)\geq |\varphi(x)-\varphi(x_0)|+
(m+1)|x-x_0|^2\geq \\|\mathbf n_0\cdot
(x-x_0)|-|D(x,x_0)|+(m+1)|x-x_0|^2\geq |\mathbf n_0\cdot
(x-x_0)|+|x-x_0|^2.
\end{split}\end{equation*}
 Introduce the orthogonal projection $P_0=\mathbb
I-\mathbf n_0\otimes \mathbf n_0$. The trivial relation
$|P_0(x-x_0)|\leq |x-x_0|$ leads to the inequality
$$
\Delta_-(x,x_0)\geq (m+1)^{-1} \big(\,|\mathbf n_0\cdot
(x-x_0)|+|P_0(x-x_0)|^2\,).
$$
We thus get
\begin{equation}\label{olga15}
\int_{B_t}\Delta_-(x,x_0)^{-2\alpha}\, dx\leq c\int_{B_t} \big(\,
|\mathbf n_0\cdot (x-x_0)| +|P_0(x-x_0)|^2\,\big)^{-2\alpha}\, dx.
\end{equation}
Now choose an orthogonal basis $(\mathbf b_i)$, $i=1,2,3$, in
$\mathbb R^3$ such that $\mathbf b_3=\mathbf n_0$. We have
$$
x-x_0=\sum_iy_i\mathbf b_i, \quad y_3=\mathbf n_0\cdot (x-x_0),
\quad |P_0(x-x_0)|^2=y_1^2+y_2^2,
$$
which yields
\begin{equation*}\begin{split}
\int_{B_t} \big(\, |\mathbf n_0\cdot (x-x_0)|
+|P_0(x-x_0)|^2\,\big)^{-2\alpha}\, dx=\int_{|y|\leq
t}(|y_3|+y_1^2+y_2^2)^{-2\alpha}\, dy\\\leq \int_{-t}^{t}
\int_{y_1^2+y_2^2<t^2}(|y_3|+y_1^2+y_2^2)^{-2\alpha}\, dy= 4\pi
\int_{0}^{t} \int_0^t(z+r^2)^{-2\alpha}\,rdr dz=\\
2\pi\int_{0}^{t} \int_0^{t^2}(z+v)^{-2\alpha}\,dzdv\leq c.
\end{split}\end{equation*}
From this and \eqref{olga15} we obtain the inequality
\begin{equation*}
\int_{B_t}\Delta_-(x,x_0)^{-2\alpha}\, dx\leq c,\
\end{equation*}
which, being combined with \eqref{olga13} and \eqref{olga14},
implies the desired estimate \eqref{olga12}.








\end{document}